\documentclass[]{amsart}
\usepackage{amssymb}
\usepackage[initials]{amsrefs}

\def\R{\mathbb R}
\def\N{\mathbb N}
\providecommand{\norm}[1]{\left\lVert#1\right\rVert}
\providecommand{\ip}[2]{\left\langle #1, #2 \right\rangle}
\providecommand{\domain}{\operatorname{dom}}
\providecommand{\range}{\operatorname{R}}
\providecommand{\F}{\operatorname{F}}

\allowdisplaybreaks
\numberwithin{equation}{section}

\theoremstyle{plain}
\newtheorem{theorem}{Theorem}[section]
\newtheorem{lemma}[theorem]{Lemma}
\newtheorem{corollary}[theorem]{Corollary}

\theoremstyle{definition}
\newtheorem{example}[theorem]{Example}

\theoremstyle{remark}
\newtheorem{remark}[theorem]{Remark}

\title[Approximation of fixed points of strongly nonexpansive sequences]
{Approximation of common fixed points of strongly nonexpansive sequences
in a Banach space}

\author{Koji~Aoyama}
\address[K.~Aoyama]
{Aoyama Mathematical Laboratory, Chiba University, 
Yayoi-cho, Inage-ku, Chiba, Chiba 263-8522, Japan}
\email{aoyama@le.chiba-u.ac.jp}

\author{Masashi~Toyoda}
\address[M.~Toyoda]
{Department of Information Science, Toho University, 
Miyama, Funabashi, Chiba 274-8510, Japan}
\email{mss-toyoda@is.sci.toho-u.ac.jp}


\begin{document}

\begin{abstract}
 The aim of this paper is to establish a strong convergence theorem for
 a strongly nonexpansive sequence in a Banach space.
 We also deal with some applications of the convergence theorem.
\end{abstract}

\maketitle

\section{Introduction}

The aim of the present paper is to study convergence of an iterative
sequence $\{x_n\}$ generated by the following algorithm:
$x_1 \in C$ and 
\begin{equation}\label{e:x_n:intro}
 x_{n+1}= \alpha_n u + (1-\alpha_n) S_n x_n
\end{equation}
for each positive integer $n$,
where $C$ is a closed convex subset of a Banach space $E$, 
$u$ is a point in $C$, 
$\{\alpha_n\}$ is a sequence in $(0,1]$, 
$\{S_n\}$ is a sequence of self-mappings of $C$ with a common fixed
point.
In particular, we prove that the sequence $\{x_n\}$ converges strongly
to a common fixed point of $\{S_n\}$ under suitable conditions. 
One of the features of our main results is that the parameter
$\{ \alpha_n \}$ in~\eqref{e:x_n:intro} is only assumed to satisfy the
following conditions:
$\lim_{n\to \infty} \alpha_n = 0$ and $\sum_{n=1}^\infty \alpha_n =
\infty$.

This paper is organized as follows:
In Section~\ref{s:pre}, we recall some definitions and known results and
we also provide some lemmas related to eventually increasing functions,
which play important roles in the proof of our main results.
In Section~\ref{s:main}, we present our main result
(Theorem~\ref{t:main}) and its corollaries (Corollaries~\ref{c:main}
and~\ref{c:single}).
Corollary~\ref{c:main} is a generalization of
\cite{MR2680036}*{Theorem~10}. 
In the final section,
we apply our main results to the common fixed point problem
for a sequence of nonexpansive mappings and the
zero point problem for an accretive operator in a Banach space.
Moreover, we also investigate convergence of a sequence $\{y_n\}$
defined by $y_1 \in C$ and 
\[
 y_{n+1}= \alpha_n f_n (y_n) + (1-\alpha_n) S_n y_n 
\]
for $n$, 
where $C$, $\{\alpha_n\}$, and $\{S_n\}$ are the same as above
and $\{f_n\}$ is a sequence of contraction-like self-mappings of $C$. 

\section{Preliminaries}\label{s:pre}

Throughout the present paper, 
$E$ denotes a real Banach space with norm $\norm{\,\cdot\,}$, 
$E^*$ the dual of $E$, 
$\ip{x}{f}$ the value of $f\in E^*$ at $x\in E$, 
$\N$ the set of positive integers, and $\R$ the set of real numbers.
The norm of $E^*$ is also denoted by $\norm{\,\cdot\,}$. 
Strong convergence of a sequence $\{x_n\}$ in $E$ to $x\in E$ is denoted
by $x_n \to x$.
The (normalized) \emph{duality mapping} of $E$ is denoted by $J$, that is, 
it is a set-valued mapping of $E$ into $E^*$ defined by 
$Jx = \bigl\{ f \in E^* : \ip{x}{f} = \norm{x}^2 = \norm{f}^2 \bigr\}$
for $x\in E$. 
It is known that 
\begin{equation}\label{ieq:|x+y|^2}
 \norm{x+y}^2 \leq \norm{x}^2 + 2\ip{y}{f}
\end{equation}
holds for all $x,y\in E$ and $f\in J(x+y)$;
see~\cite{MR1864294}*{Theorem~4.2.1}. 

Let $U_E$ denote the unit sphere of $E$, that is,
$U_E =\{ x\in E : \norm{x}=1 \}$.
The norm of $E$ is said to be
\emph{G\^{a}teaux differentiable} if the limit
\begin{equation}
\lim_{t \to 0} \frac{\norm{x+ty} -\norm{x}}t \label{eqn:norm-diff}
\end{equation}
exists for all $x,y \in U_E$; 
the norm of $E$ is said to be \emph{uniformly G\^{a}teaux differentiable}
if for each $y\in U_E$ the limit \eqref{eqn:norm-diff} is attained
uniformly for $x \in U_E$. 
We know that the duality mapping $J$ is single-valued if 
the norm of $E$ is G\^{a}teaux differentiable; 
$J$ is uniformly norm-to-weak* continuous on each bounded subset of $E$
if the norm of $E$ is uniformly G\^{a}teaux differentiable. 
A Banach space $E$ is said to be \emph{strictly convex}
if $x,y\in U_E$ and $x \ne y$ imply $\norm{x+y}<2$; 
$E$ is said to be \emph{uniformly convex} if for any $\epsilon >0$
there exists $\delta >0$ such that $x,y\in U_E$ and
$\norm{x-y}\geq \epsilon$ imply $\norm{x+y}/2 \leq 1-\delta$.
It is known that $E$ is reflexive if $E$ is uniformly convex; 
see~\cite{MR1864294} for more details. 
We know the following lemma; see~\cite{MR2581778}*{Lemma~2.1}. 

\begin{lemma}\label{l:uc-ft}
 Let $\{x_n\}$ and $\{y_n\}$ be bounded sequences in 
 a uniformly convex Banach space $E$ 
 and $\{\lambda_n\}$ a sequence in $[0,1]$ such that
 $\liminf_{n}\lambda_n>0$. 
 If $\lambda_n \norm{x_n}^2 + (1-\lambda_n)\norm{y_n}^2 
 -\norm{\lambda_n x_n + (1-\lambda_n)y_n}^2 \to 0$, 
 then $(1-\lambda_n)\norm{x_n-y_n}\to 0$. 
\end{lemma}

Let $C$ be a nonempty subset of $E$ and $T\colon C\to E$ a mapping. 
The set of fixed points of $T$ is denoted by $\F(T)$.
A mapping $T$ is said to be \emph{nonexpansive} if $\norm{Tx-Ty} \leq
\norm{x-y}$ for all $x,y \in C$;
$T$ is said to be \emph{strongly nonexpansive}~\cites{MR0470761}
if $T$ is nonexpansive and
$x_n - y_n - (T x_n - T y_n )\to 0$ whenever 
$\{x_n \}$ and $\{y_n\}$ are sequences in $C$, 
$\{x_n - y_n \}$ is bounded, 
and $\norm{x_n - y_n} - \norm{T x_n - Ty_n} \to 0$. 

We know the following lemma; see~\cite{MR3666446}*{Lemma~2.9}: 

\begin{lemma}\label{l:lemma2.9}
 Let $E$ be a Banach space whose norm is uniformly G\^{a}teaux
 differentiable,  
 $C$ a nonempty closed convex subset of $E$, 
 $T$ a nonexpansive self-mapping of $C$, 
 $\{x_n\}$ a bounded sequence in $C$, 
 $u$ a point in $C$, and
 $z_t$ a unique point in $C$ such that $z_t = tu + (1-t) Tz_t$ 
 for $t \in (0,1)$. 
 Suppose that $x_n - Tx_n \to 0$ as $n \to \infty$
 and $z_t \to w$ as $t \downarrow 0$. Then
 $\limsup_{n\to \infty} \ip{u-w}{J(x_n - w)} \leq 0$. 
\end{lemma}

A Banach space $E$ is said to have the \emph{fixed point property for
nonexpansive mappings} if every nonexpansive self-mapping of a
bounded closed convex subset $D$ of $E$ has a fixed point in $D$.

Let $C$ be a nonempty subset of $E$, 
$K$ a nonempty subset of $C$, and $Q$ a mapping of $C$ onto $K$.
Then $Q$ is said to be a \emph{retraction} if $Qx = x$ for all
$x \in K$;
$Q$ is said to be \emph{sunny} if $Q\bigl( Qx + \lambda(x-Qx) \bigr) =
Qx$ holds whenever $x\in C$, $\lambda \geq 0$,
and $Qx + \lambda(x-Qx) \in C$;
$K$ is said to be a \emph{sunny nonexpansive retract} of $C$
if there exists a sunny nonexpansive retraction~\cite{MR0328689} of $C$
onto $K$; see also~\cite{MR744194}. 

Using the results in \cite{MR576291}, \cite{MR766150},
and~\cite{MR1864294}, we obtain the following: 

\begin{lemma}\label{l:takahashi-ueda}
 Let $E$ be a reflexive Banach space, 
 $C$ a nonempty closed convex subset of $E$, and 
 $T$ a nonexpansive self-mapping of $C$ with $\F(T) \neq \emptyset$.
 Suppose that the norm of $E$ is uniformly G\^{a}teaux differentiable
 and $E$ 
 has the fixed point property for nonexpansive mappings. 
 Then $\F(T)$ is a sunny nonexpansive retract of $C$.
 Moreover, let $u$ and $z_t$ be the same as in Lemma~\ref{l:lemma2.9}. 
 Then $z_t \to Qu$ as $t \downarrow 0$, 
 where $Q$ is the unique sunny nonexpansive retraction of $C$ onto $\F(T)$. 
\end{lemma}

Let $C$ be a nonempty subset of $E$, $T\colon C\to E$ a mapping, 
and $\{S_n\}$ a sequence of mappings of $C$ into $E$. 
The set of common fixed points of $\{S_n\}$ is denoted by $\F(\{S_n\})$, 
that is, $\F(\{S_n\}) = \bigcap_{n=1}^\infty \F (S_n)$. 
Then $\{S_n\}$ is said to be a \emph{strongly nonexpansive sequence}
\cites{MR2377867, MR2581778, MR2799767} if each $S_n$ is nonexpansive and
$x_n - y_n - (S_n x_n - S_n y_n )\to 0$ whenever 
$\{x_n \}$ and $\{y_n\}$ are sequences in $C$, 
$\{x_n - y_n \}$ is bounded, 
and $\norm{x_n - y_n} - \norm{S_n x_n - S_n y_n} \to 0$; 
$\{S_n\}$ is said to satisfy the \emph{NST condition (I) with $T$}
\cites{MR2314663,MR2468414} if $\F(\{S_n\})$ is nonempty,
$\F(T) \subset \F(\{S_n\})$, and $x_n - Tx_n \to 0$ whenever
$\{x_n\}$ is a bounded sequence in $C$ and $x_n - S_n x_n \to 0$. 

\begin{remark}
 Let $C$, $T$, and $\{S_n\}$ be the same as above. 
 If $\{S_n\}$ satisfies the NST condition (I) with $T$,
 then $\F(T) \supset \F(\{S_n\})$.
 Indeed, let $z \in \F(\{S_n\})$ and set $x_n = z$ for $n \in \N$.
 Then it is clear that $\{x_n\}$ is a bounded sequence in $C$ and
 $x_n - S_n x_n = 0$ for every $n \in \N$, 
 and hence $x_n - S_n x_n \to 0$.
 Thus $z - Tz = x_n - Tx_n \to 0$, which means that $z \in \F(T)$. 
\end{remark}

\begin{example}\label{ex:const-SNS}
 Let $C$ be a nonempty subset of $E$ and $T\colon C\to E$ a mapping. 
 Set $S_n = T$ for $n \in \N$.
 Then it is clear that if $T$ is strongly nonexpansive, then $\{ S_n \}$
 is a strongly nonexpansive sequence.
 Moreover, it is also clear that if $\F(T) \ne \emptyset$,
 then $\{S_n\}$ satisfies the NST condition (I) with $T$.
\end{example}

\begin{example}[\cite{MR2581778}*{Corollary~3.8}]\label{ex:I+T_n:SNS}
 Let $C$ be a nonempty subset of a uniformly convex Banach space $E$, 
 $\{V_n\}$ a sequence of nonexpansive mappings of $C$ into
 $E$, and $\{\gamma_n\}$ a sequence in $[0,1]$ such that 
 $\liminf_{n\to\infty}\gamma_n >0$. 
 Then a sequence $\{S_n\}$ of mappings defined by 
 $S_n = \gamma_n I + (1-\gamma_n)V_n$ for $n\in\N$
 is a strongly nonexpansive sequence, 
 where $I$ is the identity mapping on $C$. 
\end{example}

We know the following lemma; see~\cites{MR1911872,MR2338104}. 

\begin{lemma}\label{lm:seq}
 Let $\{ \xi_n \}$ be a sequence of nonnegative real numbers, 
 $\{ \gamma_n \}$ a sequence of real numbers, 
 and $\{ \alpha_n \}$ a sequence in $[0,1]$. 
 Suppose that 
 $\xi_{n+1} \leq (1- \alpha_n) \xi_n + \alpha_n \gamma_n$
 for every $n\in\N$, $\limsup_{n\to\infty} \gamma_n \leq 0 $, 
 and $\sum_{n=1}^\infty \alpha_n = \infty$. Then $\xi_n \to 0$. 
\end{lemma}

We use the following lemma in Section~\ref{s:app}. 

\begin{lemma}\label{l:absolutely}
 Let $E$ be a Banach space,
 $\{ y_j^n \}$ a double sequence in $E$ indexed by $(j,n) \in \N \times
 \N$, and $\{\lambda_j \}$ a sequence in $[0,1]$. 
 Suppose that 
 $\sup \bigl\{ \bigl\lVert y^n_j \bigr\rVert :
 (j,n)\in \N \times \N \bigr \}$ is finite, 
 $\sum_{j=1}^\infty \lambda_j = 1$, 
 and $\lim_{n\to \infty} y^n_j = 0$ for all $j \in \N$.
 Then $\sum_{j=1}^\infty \lambda_j y_j^n$ converges (absolutely)
 for all $n \in \N$ and
 $\lim_{n\to 0}\sum_{j=1}^\infty \lambda_j y_j^n = 0$. 
\end{lemma}

\begin{proof}
 Let $n \in \N$ be fixed.
 Set $M= \sup \bigl\{ \bigl\lVert y^n_j \bigr\rVert :
 (j,n) \in \N \times \N \bigr \}$. 
 Without loss of generality, we may assume that $M>0$. 
 Then 
 \[
 \sum_{j=1}^m \norm{\lambda_j y_j^n}
 = \sum_{j=1}^m \lambda_j \norm{y_j^n}
 \leq M \sum_{j=1}^m \lambda_j \leq M
 \]
 holds for all $m \in \N$.
 Thus $\sum_{j=1}^\infty \lambda_j y_j^n$ converges absolutely
 for all $n \in \N$. 

 We next show that $\sum_{j=1}^\infty \lambda_j y_j^n \to 0$ as $n \to
 \infty$. 
 Let $\epsilon> 0$ be fixed. 
 Then there exists $m \in \N$ such that
 $\sum_{j=m+1}^\infty \lambda_j < \epsilon /(2M)$
 by virtue of $\sum_{j=1}^\infty \lambda_j = 1$. 
 On the other hand,
 since $\lim_{n\to \infty} y^n_j = 0$ for all $j \in \N$,
 there exists $l \in \N$ such that
 $\bigl\lVert y_j^n \bigr\rVert < \epsilon/2$ for all $n \in \N$ with $n\geq l$
 and $j \in \{1, \dotsc, m\}$.
 Consequently, we see that 
 \[
  \sum_{j=1}^\infty \norm{\lambda_j y_j^n}  
  = \sum_{j=1}^m \lambda_j \norm{y_j^n}
  + \sum_{j=m+1}^\infty \lambda_j \norm{y_j^n}
  \leq \sum_{j=1}^m \lambda_j \dfrac{\epsilon}2
  + \sum_{j=m+1}^\infty \lambda_j M  < \epsilon
 \]
 for all $n \in \N$ with $n \geq l$.
 Therefore, $\lim_{n\to 0}\sum_{j=1}^\infty \lambda_j y_j^n = 0$. 
\end{proof}

In the rest of this section, we provide some lemmas about an eventually
increasing function and a strongly nonexpansive sequence. 

A function $\tau \colon \N \to \N$ is said to be \emph{eventually
increasing} \cites{MR2960628,MR3213161,NMJ2016,pNACA2011} 
or \emph{unboundedly increasing}~\cite{MR3666446}
if $\lim_{n\to\infty}\tau(n) = \infty$ and 
$\tau(n) \leq \tau(n+1)$ for all $n\in \N$. 
It is clear that if $\tau \colon \N \to \N$ is an eventually increasing
function and $\{\xi_n\}$ is a sequence of real numbers such that
$\xi_n \to 0$, then $\xi_{\tau(n)}\to 0$. 

Using \cite{MR2466027}*{Lemma 3.1}, we obtain the following: 

\begin{lemma}[\cite{MR2960628}*{Lemma 3.4}]\label{l:neo_mainge}
 Let $\{\xi_n\}$ be a sequence of nonnegative real numbers which is not
 convergent. 
 Then there exist $N\in \N$ and an eventually increasing function 
 $\tau \colon \N \to \N$ such that
 $\xi_{\tau(n)} \le \xi_{\tau(n)+1}$ for all $n \in \N$ and 
 $\xi_n \le \xi_{\tau(n)+1}$ for all $n \geq N$. 
\end{lemma}

In Lemma~\ref{l:neo_mainge}, we cannot replace the word ``eventually''
by ``strictly''; see \cite{MR2960628}*{Example 3.3}. 

Applying Lemma~\ref{l:neo_mainge}, we obtain the following lemma,
which is used in Section~\ref{s:main}. 

\begin{lemma}\label{l:seq+mainge}
 Let $\{ \xi_n \}$ be a sequence of nonnegative real numbers, 
 $\{ \alpha_n \}$ a sequence in $(0,1]$, and 
 $\{ \gamma_n \}$ a sequence of real numbers. 
 Suppose that 
 $\xi_{\tau(n)+1} \leq (1- \alpha_{\tau(n)}) \xi_{\tau(n)}
 + \alpha_{\tau(n)} \gamma_{\tau(n)}$ for all $n \in \N$
 and $\limsup_{n\to\infty} \gamma_{\tau(n)} \leq 0$
 whenever $\tau \colon \N \to \N$ is an eventually increasing function
 such that $\xi_{\tau(n)} \leq \xi_{\tau(n)+1}$ for all $n \in \N$. 
 Then $\{ \xi_n \}$ is convergent. 
\end{lemma}

\begin{proof}
 Assume that $\{ \xi_n \}$ is not convergent. 
 Then Lemma~\ref{l:neo_mainge} implies that there exist $N\in \N$ and an
 eventually increasing function $\tau\colon \N \to \N$ such that
 $\xi_{\tau(n)} \le \xi_{\tau(n)+1}$ for all $n\in \N$ and 
 $\xi_n \le \xi_{\tau(n)+1}$ for all $n \geq N$.
 Thus it follows that 
 \[
 \xi_{\tau(n)+1} \leq (1- \alpha_{\tau(n)}) \xi_{\tau(n)}
 + \alpha_{\tau(n)} \gamma_{\tau(n)}  
 \leq (1- \alpha_{\tau(n)}) \xi_{\tau(n)+1}
 + \alpha_{\tau(n)} \gamma_{\tau(n)}
 \]
 for all $n \in \N$. 
 Hence we deduce from $\alpha_{\tau(n)} > 0$ that
 $\xi_{\tau(n)+1} \leq \gamma_{\tau(n)}$ for all $n \in \N$. 
 Therefore, by assumption, we conclude that
 \[
 0 < \limsup_{n\to\infty} \xi_n
 \leq \limsup_{n\to\infty} \xi_{\tau(n) + 1}
 \leq \limsup_{n\to\infty} \gamma_{\tau(n)} \leq 0 ,
 \]
 which is a contradiction. 
\end{proof}

The following lemma plays a crucial role in Section~\ref{s:main}. 

\begin{lemma}\label{l:tau-eif}
 Let $E$ be a Banach space, $C$ a nonempty subset of $E$, 
 $\{f_n\}$ a sequence of mappings of $C$ into $\R$, 
 $\{g_n\}$ a sequence of mappings of $C$ into $[0,\infty)$, 
 $\tau\colon \N \to \N$ an eventually increasing function, 
 and $\{z_n\}$ a bounded sequence in $C$. 
 Suppose that $f_{\tau(n)}(z_n) \to 0$,
 there exits $p \in C$ such that $f_n (p) = 0$ for all $n \in \N$,
 and $g_n(x_n) \to 0$ whenever $\{x_n\}$ is a bounded
 sequence in $C$ such that $f_n (x_n) \to 0$. 
 Then $g_{\tau(n)}(z_n) \to 0$. 
\end{lemma}

\begin{proof}
 Assume that $g_{\tau(n)} (z_n) \not\to 0$.
 Then there exist $\epsilon > 0$  and a strictly increasing function 
 $\sigma \colon \N \to \N$ such that
 $\tau \circ \sigma$ is also strictly increasing and
 \begin{equation}\label{geq_e}
  g_{\tau\circ \sigma(n)}(z_{\sigma(n)}) \geq \epsilon
 \end{equation}
 for all $n \in \N$. Set $\mu = \tau \circ \sigma$ and
 $\range (\mu) = \{\mu(n): n \in \N\}$.  
 Let $\{y_n\}$ be a sequence in $C$ defined by 
 \[
 y_n = \begin{cases}
	z_{\sigma \circ \mu^{-1}(n)} & \text{if $n \in \range (\mu)$};\\
	p & \text{if $n \notin \range (\mu)$}
       \end{cases}
 \]
 for  $n\in \N$. 
 It is clear that $\{ y_n \}$ is bounded, 
 $f_n (y_n)
 = f_{\tau (\sigma \circ \mu^{-1}(n))} ( z_{\sigma \circ\mu^{-1}(n)})$
 for $n \in \range (\mu)$, and $f_n (y_n) = 0$ for $n \notin \range (\mu)$. 
 Since $\sigma \circ \mu^{-1}$ is strictly increasing, 
 we see that $f_n (y_n) \to 0$, and hence $g_n (y_n) \to 0$ by assumption. 
 Since $\mu$ is strictly increasing, it follows from 
 $y_{\mu(n)} = z_{\sigma(\mu^{-1}(\mu(n)))} = z_{\sigma(n)}$ that
 \[
 g_{\tau\circ\sigma (n)} (z_{\sigma(n)})
 = g_{\mu (n)} (y_{\mu (n)}) \to 0, 
 \]
 which contradicts to \eqref{geq_e}. 
\end{proof}

\section{Main results}\label{s:main}

In this section, we prove the following strong convergence theorem and
derive its corollaries. 

\begin{theorem}\label{t:main}
 Let $E$ be a reflexive Banach space whose norm is uniformly G\^{a}teaux
 differentiable, $C$ a nonempty closed convex subset of $E$, 
 $\{S_n\}$ a sequence of self-mappings of $C$, 
 $T$ a nonexpansive self-mapping of $C$, 
 $\{\alpha_n\}$ a sequence in $(0,1]$, $u$ a point in $C$,
 and $\{x_n\}$ a sequence defined by $x_1 \in C$ and 
 \begin{equation}\label{e:x_n}
  x_{n+1}= \alpha_n u + (1-\alpha_n) S_n x_n
 \end{equation}
 for $n\in\N$.
 Suppose that $E$ has the fixed point property for nonexpansive
 mappings, 
 $\alpha_n \to 0$, $\sum_{n=1}^\infty \alpha_n = \infty$,
 $\{S_n\}$ is a strongly nonexpansive sequence, 
 and $\{S_n\}$ satisfies the NST condition (I) with $T$.
 Then $\{x_n\}$ converges strongly to $Qu$, where $Q$ is the sunny
 nonexpansive retraction of $C$ onto $\F(T)$.
\end{theorem}

To show Theorem~\ref{t:main}, we need the lemmas below: 

\begin{lemma}\label{l:tau}
 Let $E$ be a Banach space, $C$ a nonempty subset of $E$, 
 $\tau\colon \N \to \N$ an eventually increasing function, 
 and $\{S_n\}$ a sequence of mappings of $C$ into $E$ such that
 $\F (\{S_n\})$ is nonempty.
 Then the following hold:
 \begin{enumerate}
  \item If $\{S_n\}$ is a strongly nonexpansive sequence and 
	$\{z_n\}$ is a bounded sequence in $C$ such that 
	$\norm{z_n - p} - \norm{S_{\tau(n)} z_n - p} \to 0$
	for some $p \in \F (\{S_n\})$, then 
	$S_{\tau(n)} z_n - z_n \to 0$. 
  \item	If $\{S_n\}$ satisfies the NST condition (I) with a mapping
	$T\colon C \to E$, then so does $\{ S_{\tau(n)} \}$. 
 \end{enumerate}
\end{lemma}

\begin{proof}
 We first show (1). 
 Let $f_n \colon C \to \R$ and $g_n \colon C \to [0,\infty)$ be functions
 defined by
 $f_n (x) = \norm{x-p} - \norm{S_n x - p}$
 and 
 $g_n (x) = \norm{S_n x - x}$ for all $x \in C$ and $n \in \N$,
 respectively. 
 Then $f_{\tau(n)}(z_n) \to 0$ and $f_n (p) = 0$ for all $n \in \N$. 
 Since $\{S_n\}$ is a strongly nonexpansive sequence, it follows that
 $g_n(x_n) \to 0$ whenever $\{x_n\}$ is a bounded sequence in $C$ and
 $f_n(x_n) \to 0$. Therefore Lemma~\ref{l:tau-eif} implies that 
 $\norm{S_{\tau(n)} z_n - z_n} = g_{\tau(n)} (z_n)\to 0$. 

 We next show (2). It is clear that
 $\F(\{S_n\}) \subset \F(S_{\tau(n)})$ for every $n \in \N$. 
 Thus $\F(\{S_n\}) \subset \F\bigl( \{ S_{\tau(n)} \}\bigr)$. 
 Since $\{S_n\}$ satisfies the NST condition (I) with $T$,
 we know that $\F \bigl( \{ S_n \} \bigr)$ is nonempty and
 $\F(T) \subset \F \bigl( \{ S_n \} \bigr)$. 
 Hence $\F \bigl( \{ S_{\tau(n)} \} \bigr)$ is nonempty and
 $\F(T) \subset \F \bigl( \{S_{\tau(n)}\} \bigr)$. 
 Let $\{z_n\}$ be a bounded sequence in $C$ such that
 $z_n - S_{\tau(n)} z_n \to 0$. It is enough to show that $z_n - T z_n \to 0$. 
 Let $f_n \colon C \to \R$ and $g_n \colon C \to [0,\infty)$ be functions
 defined by $f_n (x) = \norm{S_n x - x}$ and 
 $g_n (x) = \norm{Tx - x}$ for all $x \in C$ and $n \in \N$, respectively. 
 Let $p \in \F(\{S_n\})$. 
 Then $f_{\tau(n)}(z_n) \to 0$ and $f_n (p) = 0$ for all $n \in \N$.
 Since $\{S_n\}$ satisfies the NST condition (I) with $T$, it follows that 
 $g_n(x_n) \to 0$ whenever $\{x_n\}$ is a bounded sequence in $C$ such that
 $f_n(x_n) \to 0$. Therefore Lemma~\ref{l:tau-eif} implies that
 $\norm{z_n - T z_n} = g_{\tau(n)} (z_n)\to 0$. 
\end{proof}

The proof of Theorem~\ref{t:main} is based on the following lemma: 

\begin{lemma}\label{l:main}
 Let $E$ be a Banach space whose norm is uniformly G\^{a}teaux
 differentiable. 
 Let $C$, $\{S_n\}$, $T$, $\{\alpha_n\}$, $u$, and $\{x_n\}$ 
 be the same as in Theorem~\ref{t:main}. 
 Suppose that $z_t \to w \in \F(T)$ as $t \downarrow 0$,
 where $z_t$ is a unique point in $C$ such that $z_t = tu + (1-t) Tz_t$
 for $t \in (0,1)$.
 Then $\{x_n\}$ converges strongly to $w$. 
\end{lemma}

\begin{proof} 
 By the assumption that $\{S_n\}$ satisfies the NST condition (I) with
 $T$, we see that $w \in \F(S_n)$ for all $n \in \N$.
 Since $S_n$ is nonexpansive, it follows from~\eqref{e:x_n} that
 \begin{align}
  \norm{x_{n+1} - w} \label{basic1}
  &\leq \alpha_n \norm{u-w} + (1-\alpha_n) \norm{S_n x_n - w} \\
  &\leq \alpha_n \norm{u-w} + (1-\alpha_n) \norm{x_n - w} \notag
 \end{align}
 for all $n\in \N$. 
 Thus, by induction on $n$, we have
 \[
  \norm{S_n x_n - w} \leq \norm{x_n - w} 
 \leq \max \{ \norm{u-w}, \norm{x_1 - w} \}. 
 \]
 Therefore, $\{x_n\}$ and $\{S_n x_n\}$ are bounded. 
 Hence, by the assumption that $\alpha_n \to 0$, we see that 
 \begin{equation}\label{x_{n+1}-S_nx_n-to0}
  x_{n+1} - S_n x_n = \alpha_n (u - S_n x_n)\to 0.
 \end{equation}
 We also deduce from~\eqref{ieq:|x+y|^2} that 
 \begin{align}
  \begin{split} \label{basic2}
   \norm{x_{n+1} - w}^2 
   &= \norm{\alpha_n (u-w) + (1-\alpha_n)(S_n x_n - w)}^2 \\
   &\leq (1-\alpha_n)^2 \norm{S_n x_n - w}^2 
   + 2 \alpha_n \ip{u-w}{J(x_{n+1} - w)}\\
   &\leq (1-\alpha_n) \norm{x_n - w}^2 
   + 2 \alpha_n \ip{u-w}{J(x_{n+1} - w)} 
  \end{split}
 \end{align}
 for all $n\in \N$. 

 We next show that $\{ \norm{x_n - w} \}$ is convergent by using
 Lemma~\ref{l:seq+mainge}.
 Let $\tau \colon \N \to \N$ be an eventually increasing function. 
 Suppose that $\norm{x_{\tau(n)} - w} \le \norm{x_{\tau(n)+1} - w}$
 for all $n \in \N$. 
 Since $S_{\tau(n)}$ is nonexpansive, $w \in \F(S_{\tau(n)})$,
 and $\alpha_{\tau(n)}\to 0$, 
 it follows from~\eqref{basic1} that 
 \begin{align*}
  0 &\leq \norm{x_{\tau(n)} - w} - \norm{S_{\tau(n)} x_{\tau(n)} -w} \\
  &\leq \norm{x_{\tau(n)+1} - w} - \norm{S_{\tau(n)} x_{\tau(n)} -w}
  \leq \alpha_{\tau(n)} \norm{u-w} \to 0
 \end{align*}
 as $n\to \infty$. 
 Thus Lemma~\ref{l:tau} implies that
 \begin{equation}\label{135108_16Feb16}
  S_{\tau(n)} x_{\tau(n)} -  x_{\tau(n)} \to 0.
 \end{equation}
 Since $\{S_n\}$ satisfies the NST condition (I) with $T$,
 it follows~from Lemma~\ref{l:tau} that
 $x_{\tau(n)} - T x_{\tau(n)} \to 0$. 
 Consequently, Lemma~\ref{l:lemma2.9} implies that
 \begin{equation}\label{eq:limsup_0-tau_n}
  \limsup_{n\to \infty} \ip{u-w}{J(x_{\tau(n)} - w)} \leq 0. 
 \end{equation}
 Since $J$ is uniformly norm-to-weak* continuous on a bounded set and 
 \[
 \norm{x_{\tau(n)+1}- x_{\tau(n)}}
 \leq \norm{x_{\tau(n)+1} - S_{\tau(n)}x_{\tau(n)}} +
 \norm{S_{\tau(n)} x_{\tau(n)}- x_{\tau(n)}}  \to 0
 \]
 by~\eqref{x_{n+1}-S_nx_n-to0} and~\eqref{135108_16Feb16},
 it turns out that 
 \begin{equation}\label{eq:uG}
  \ip{u-w}{J(x_{\tau(n) + 1} -w) - J(x_{\tau(n)} - w)} \to 0.   
 \end{equation}
 Thus, by virtue of~\eqref{eq:limsup_0-tau_n} and \eqref{eq:uG},
 we have 
 \begin{equation}\label{eq:limsup_0-tau_n+1}
  \limsup_{n\to \infty} \ip{u-w}{J(x_{\tau(n)+1} - w)} \leq 0. 
 \end{equation}
 According to~\eqref{basic2} and \eqref{eq:limsup_0-tau_n+1}, we
 conclude from Lemma~\ref{l:seq+mainge} that $\{ \norm{x_n -w}\}$ is
 convergent.

 We lastly show that $x_n \to w$. 
 Since $S_n$ is nonexpansive, $w \in \F(S_n)$,
 and $\{ \norm{x_n -w}\}$ is convergent, 
 it follows from~\eqref{basic1} and $\alpha_n \to 0$ that 
 \begin{align*}
  0 &\leq \norm{x_n - w} - \norm{S_n x_n - w}
  \leq \norm{x_n - w} - (1-\alpha_n)\norm{S_n x_n - w}  \\
  &\leq \norm{x_n - w} - \norm{x_{n+1} - w} + \alpha_n \norm{u - w}
  \to 0.
 \end{align*}
 Thus $x_n - S_n x_n \to 0$ by the assumption that 
 $\{S_n\}$ is a strongly nonexpansive sequence,
 and hence $x_n - Tx_n \to 0$
 by the assumption that $\{S_n\}$ satisfies the NST condition (I) with
 $T$. 
 Consequently, Lemma~\ref{l:lemma2.9} shows that 
 \[
 \limsup_{n\to \infty} \ip{u-w}{J(x_{n+1} - w)}
 = \limsup_{n\to \infty} \ip{u-w}{J(x_n - w)} \leq 0. 
 \]
 Taking into account $\sum_{n=1}^\infty \alpha_n = \infty$
 and~\eqref{basic2}, we deduce from Lemma~\ref{lm:seq} that
 $\norm{x_n - w}\to 0$. This completes the proof. 
\end{proof}

\begin{proof}[Proof of Theorem~\ref{t:main}]
 Let $z_t$ be the same as in Lemma~\ref{l:main} for $t \in (0,1)$. 
 Then it follows from Lemma~\ref{l:takahashi-ueda} that 
 $z_t \to Qu$ as $t \downarrow 0$.
 Therefore Lemma~\ref{l:main} implies the conclusion. 
\end{proof}

It is known that if a Banach space $E$ is uniformly convex, then $E$ is
reflexive and has the fixed point property for nonexpansive mappings;
see~\cite{MR1864294}.
Thus the following corollary is a direct consequence of
Theorem~\ref{t:main}. 

\begin{corollary}\label{c:main}
 Let $E$ be a uniformly convex Banach space 
 whose norm is uniformly G\^{a}teaux differentiable. 
 Let $C$, $\{S_n\}$, $T$, $\{\alpha_n\}$, $u$, $\{x_n\}$, and $Q$ be the
 same as in Theorem~\ref{t:main}.
 Then $\{x_n\}$ converges strongly to $Qu$. 
\end{corollary}

\begin{remark}
 Corollary~\ref{c:main} is a generalization
 of~\cite{MR2680036}*{Theorem~10}.
 Indeed, under the assumptions of \cite{MR2680036}*{Theorem 10},
 we can verify that $\{T_n\}$ satisfies the NST condition (I) with $T$
 by using \cite{MR2338104}*{Lemma 3.2}.  
\end{remark}

Using Theorem~\ref{t:main} and Example~\ref{ex:const-SNS},
we obtain the following corollary: 

\begin{corollary}\label{c:single}
 Let $E$, $C$, $\{\alpha_n\}$, and $u$ be the same as in
 Theorem~\ref{t:main}. Let $T$ be a strongly nonexpansive self-mapping
 of $C$ and $\{x_n\}$ a sequence defined by $x_1 \in C$ and
 $x_{n+1}= \alpha_n u + (1-\alpha_n) T x_n$
 for $n\in\N$.
 Suppose that $\F(T)$ is nonempty. 
 Then $\{x_n\}$ converges strongly to $Qu$, where
 $Q$ is the sunny nonexpansive retraction of $C$ onto $\F(T)$. 
\end{corollary}

\begin{remark}
 Corollary~\ref{c:single} is related
 to~\cite{MR2680036}*{Theorem~4~(1)}. 
 Indeed, it is known that a uniformly smooth Banach space
 is reflexive and has the fixed point property for nonexpansive
 mappings, and moreover,
 the norm is uniformly G\^{a}teaux differentiable;
 see~\cites{MR1864294,MR689566,MR1074005}.
\end{remark}

\section{Applications} \label{s:app}

In this section, we deal with three applications of Theorem~\ref{t:main} and
Corollary~\ref{c:main}.

\subsection{Common fixed points of a sequence of nonexpansive mappings}

We first address the problem of approximating a common fixed point of a
sequence of nonexpansive mappings. Using Corollary~\ref{c:main},
we obtain the following theorem: 

\begin{theorem}\label{t:CFPP}
 Let $E$ be a uniformly convex Banach space whose norm is uniformly
 G\^{a}teaux differentiable,
 $C$ a nonempty closed convex subset of $E$, 
 $\{T_n\}$ a sequence of nonexpansive self-mappings of $C$, 
 $\{\alpha_n\}$ the same as in Theorem~\ref{t:main},
 $\{\beta_n^k\}$ a double sequence in $(0,1]$ indexed by
 $n \in \N$ and $k \in \{1,\dotsc,n\}$, 
 and $\{\gamma_n\}$ a sequence in $(0,1)$. 
 Suppose that $\F(\{T_n\})$ is nonempty, 
 $\sum_{k=1}^n \beta_n^k = 1$ for all $n \in \N$, 
 $\inf\{ \beta_n^k: n\geq k\} > 0$ for all $k \in \N$, 
 $\inf_n \gamma_n > 0$, and $\sup_n \gamma_n < 1$. 
 Let $u$ be a point in $C$
 and $\{x_n\}$ a sequence defined by $x_1 \in C$ and
 \[
 x_{n+1}= \alpha_n u + (1-\alpha_n) \biggl[
 \gamma_n x_n + (1- \gamma_n) \sum_{k=1}^n \beta_n^k T_k x_n \biggr]
 \]
 for $n\in\N$.
 Then $\{x_n\}$ converges strongly to $Qu$, where $Q$ is the
 sunny nonexpansive retraction of $C$ onto $\F(\{T_n\})$.
\end{theorem}

In order to prove Theorem~\ref{t:CFPP}, we need the lemmas below: 

\begin{lemma}[\cite{MR0324491}*{Lemma 3} and its proof]\label{l:bruck}
 Let $E$ be a strictly convex  Banach space, 
 $C$ a nonempty closed convex subset of $E$, 
 $\{T_n\}$ a sequence of nonexpansive mappings of $C$ into $E$, 
 and $\{\lambda_n \}$ a sequence in $(0,1)$. 
 Suppose that $\F( \{T_n \})$ is nonempty
 and $\sum_{n=1}^\infty \lambda_n = 1$. 
 Then $T=\sum_{n=1}^\infty \lambda_n T_n$ is well-defined and
 nonexpansive, and moreover, $\F( \{T_n \}) = \F(T)$. 
\end{lemma}

\begin{lemma}\label{l:check-NST}
 Let $E$ be a uniformly convex Banach space, 
 $C$ a nonempty closed convex subset of $E$, 
 $\{T_n\}$ a sequence of nonexpansive mappings of $C$ into $E$,
 $\{\beta_n^k\}$ the same as in Theorem~\ref{t:CFPP}, 
 $\{\lambda_n\}$ a sequence in $(0,1)$, 
 and $V_n$ a mapping of $C$ into $E$ defined by
 $V_n =  \sum_{k=1}^n \beta_n^k T_k$ for $n \in \N$. 
 Suppose that $\F(\{ T_n \})$ is nonempty
 and $\sum_{n=1}^\infty \lambda_n = 1$. 
 If $\{x_n\}$ is a bounded sequence in $C$ such that $x_n - V_n x_n \to 0$,
 then $\lim_{n\to \infty} \norm{x_n - T_j x_n} = 0$ for all $j \in \N$.
 Moreover, $\{V_n\}$ satisfies the NST condition (I) with a mapping
 $T\colon C\to E$ defined by $T = \sum_{n=1}^\infty \lambda_n T_n$. 
\end{lemma}

\begin{proof}
 Let $j \in \N$ and $z \in \F(\{T_n\})$ be fixed. 
 Set $N_n = \{i \in \N: 1 \leq i \leq n,\, i \ne j\}$ for $n \in \N$. 
 Let $U_n$ be a mapping of $C$ into $E$ defined by
 \[
 U_n = \dfrac{1}{1-\beta_n^j} \sum_{k \in N_n} \beta_n^k T_k
 \]
 for $n \in \N$ with $n > j$. 
 Then one can verify that $V_n = \beta_n^j T_j + (1- \beta_n^j) U_n$, 
 $U_n$ is nonexpansive, and $z = U_n z$ for all $n \in \N$
 with $n > j$. 
 Since $\norm{\,\cdot\,}^2$ is convex,
 both $T_j$ and $V_n$ are nonexpansive,
 $z = T_j z = V_n z$, $\{x_n\}$ is bounded,
 and $\norm{x_n -V_n x_n} \to 0$, it turns out that
 \begin{align*}
  0 &\leq
  \beta_n^j \norm{z-T_j x_n}^2 + (1- \beta_n^j) \norm{z-U_n x_n}^2 \\
  &\qquad - \norm{\beta_n^j (z-T_j x_n) + (1- \beta_n^j)(z-U_n x_n)}^2\\
  &\leq \norm{z-x_n}^2 - \norm{z-V_n x_n}^2\\
  &\leq 2\norm{z-x_n}(\norm{z-x_n} - \norm{z-V_n x_n})\\
  &\leq 2\norm{z-x_n}\norm{x_n -V_n x_n} \to 0
 \end{align*}
 as $n \to \infty$. 
 Hence Lemma~\ref{l:uc-ft} implies that
 \[
 0 \leq (1 - \beta_n^j) \norm{U_n x_n - T_j x_n}
 = (1 - \beta_n^j) \norm{z - T_j x_n - (z - U_n x_n)} \to 0
 \]
 as $n \to \infty$. 
 Since $V_n - T_j = (1- \beta_n^j)(U_n - T_j)$ and $x_n - V_n x_n \to 0$, 
 we deduce that
 \[
 \norm{x_n - T_jx_n} \leq
 \norm{x_n -V_nx_n} + (1- \beta_n^j) \norm{U_n x_n - T_jx_n} \to 0
 \]
 as $n\to \infty$. Therefore, 
 $\lim_{n\to \infty} \norm{x_n - T_j x_n} = 0$ for all $j \in \N$. 

 We next show that $\{V_n\}$ satisfies the NST condition (I) with $T$. 
 By Lemma~\ref{l:bruck}, we know that $\F(T) = \F(\{T_n\})$. 
 It is clear from the definition of $V_n$ that 
 $\F(\{T_n\}) \subset \F(V_n)$ for all $n \in \N$.
 Thus $\F(\{T_n\}) \subset \F(\{V_n\})$,
 and hence $\F(\{ V_n \})$ is nonempty and 
 $\F(T) \subset \F(\{V_n\})$. 
 Let $\{y_n\}$ be a bounded sequence in $C$ such that $y_n - V_n y_n \to
 0$. Since 
 $\lim_{n\to \infty} \norm{y_n - T_j y_n} = 0$ for all $j \in \N$
 from the earlier part of this proof, it follows from
 Lemma~\ref{l:absolutely} that 
 \[
 \norm{y_n - Ty_n} =
 \biggl\lVert \sum_{j=1}^\infty \lambda_j (y_n - T_j y_n) \biggr\rVert \to 0
 \]
 as $n \to \infty$. This completes the proof. 
\end{proof}

\begin{lemma}\label{l:NST:cc}
 Let $E$ be a Banach space, 
 $C$ a nonempty subset of $E$, 
 $\{V_n\}$ a sequence of mappings of $C$ into $E$,
 $\{ \gamma_n\}$ a sequence in $\R$, 
 and $\{S_n\}$ a sequence of mappings of $C$ into $E$
 defined by $S_n = \gamma_n I + (1- \gamma_n)V_n$ for $n \in \N$,
 where $I$ is the identity mapping on $C$. 
 Suppose that $\F(\{V_n\})$ is nonempty and $\sup_n \gamma_n < 1$. 
 If $\{V_n\}$ satisfies the NST condition (I) with a mapping
 $T\colon C \to E$, then so does $\{ S_n \}$. 
\end{lemma}

\begin{proof}
 It is easy to check that $\F(S_n) = \F(V_n)$ for every $n \in \N$.
 Thus $\F(\{S_n\}) = \F(\{V_n\})$.
 Since $\{V_n\}$ satisfies the NST condition (I) with $T$,
 we know that $\F(\{V_n\})$ is nonempty and $\F(T) \subset \F(\{V_n\})$.
 Therefore, $\F(\{S_n\})$ is nonempty and $\F(T) \subset \F(\{S_n\})$.
 Let $\{x_n\}$ be a bounded sequence in $C$ such that $x_n - S_n x_n \to
 0$. Taking into account $I - S_n = (1-\gamma_n)(I - V_n)$
 and $\sup_n \gamma_n <1$, we deduce that 
 \[
  0 \leq (1- \sup\nolimits_{n} \gamma_n) \norm{x_n - V_n x_n}
 \leq (1- \gamma_n) \norm{x_n - V_n x_n}
 = \norm{x_n - S_n x_n} \to 0
 \]
 as $n \to \infty$. Thus $x_n - V_n x_n \to 0$. 
 Since $\{V_n\}$ satisfies the NST condition (I) with $T$,
 it follows that $x_n - Tx_n \to 0$. 
 As a result, $\{ S_n \}$ satisfies the NST condition (I) with $T$. 
\end{proof}

Using Corollary~\ref{c:main} and lemmas above, we can prove
Theorem~\ref{t:CFPP}.

\begin{proof}[Proof of Theorem~\ref{t:CFPP}]
 Let $T$ be a mapping defined by $T = \sum_{n=1}^\infty T_n/2^n$.
 Then Lemma~\ref{l:bruck} shows that $T$ is nonexpansive and $\F(T) =
 \F(\{T_n\})$.
 Moreover, it is clear that $T$ is a self-mapping of $C$. 

 Let $S_n$ and $V_n$ be mappings defined by
 $S_n = \gamma_n I + (1-\gamma_n) \sum_{k=1}^n \beta_n^k T_k$
 and $V_n = \sum_{k=1}^n \beta_n^k T_k$
 for $n \in \N$, where $I$ is the identity mapping on $C$.
 Then it is obvious that each $S_n$ is a self-mapping of $C$
 and~\eqref{e:x_n} holds for all $n \in \N$.
 Since each $T_k$ is nonexpansive and $\sum_{k=1}^n \beta_n^k = 1$,
 it follows that $V_n$ is nonexpansive for every $n \in \N$.
 Thus Example~\ref{ex:I+T_n:SNS} implies that
 $\{S_n\}$ is a strongly nonexpansive sequence. 
 On the other hand, it follows from Lemma~\ref{l:check-NST} that
 $\{V_n\}$ satisfies the NST condition (I) with $T$.
 Thus Lemma~\ref{l:NST:cc} implies that
 $\{S_n\}$ satisfies the NST condition (I) with $T$.
 Consequently, Corollary~\ref{c:main} implies that $x_n \to Qu$. 
\end{proof}

\subsection{Zeros of accretive operators}

We next consider the problem of finding a zero of an accretive operator in a
Banach space and prove a convergence theorem for the problem. 

Let $A$ be a set-valued mapping of $E$ into $E$. 
The domain of $A$ is denoted by $\domain(A)$, 
the range of $A$  by $\range (A)$, 
and the set of zeros of $A$ by $A^{-1}0$, that is, 
$\domain(A) = \{x \in E: Ax \ne \emptyset\}$,
$\range(A) = \bigcup \{Ax\colon x\in \domain (A)\}$,
and $A^{-1}0 = \{x \in \domain(A): 0 \in Ax \}$. 
We say that $A$ is an \emph{accretive} operator on $E$ if for
$x,y \in \domain (A)$, $u \in Ax$, and $v \in Ay$ there exists
$j \in J(x-y)$ such that $\ip{u - v}{j} \geq 0$. 

Let $A$ be an accretive operator on $E$, 
$I$ the identity mapping on $E$, and $\lambda$ a positive real number.
It is known that $(I+ \lambda A)^{-1}$ is a single-valued
mapping of $\range(I+ \lambda A)$ onto $\domain (A)$.
The mapping $(I+ \lambda A)^{-1}$ is called the \emph{resolvent} of $A$ and
is denoted by $J_\lambda$.
It is also known that $\F(J_\lambda) = A^{-1}0$ and 
$A^{-1}0$ is a sunny nonexpansive retract of $E$
under the assumptions of Lemma~\ref{l:takahashi-ueda};
see \cite{MR1864294}. 

Using Corollary~\ref{c:main}, we obtain the following theorem;
see~\cites{MR1788273, MR1758838, MR1802240, MR2218890, MR2338104,
MR2468414, MR2680036, MR2719910} for related results. 

\begin{theorem}[\cite{MR3666446}*{Theorem 3.1}]
 Let $E$, $C$, $\{\alpha_n\}$, and $u$ be the same as in
 Corollary~\ref{c:main},
 $A$ an accretive operator on $E$, 
 and $\{\lambda_n\}$ a sequence of positive real numbers. 
 Suppose that $A^{-1}0$ is nonempty,
 $\overline{\domain(A)} \subset C \subset 
 \range(I+\lambda A)$ for all $\lambda > 0$, 
 and $\inf_n \lambda_n > 0$, where
 $\overline{\domain (A)}$ is the closure of $\domain(A)$ and 
 $I$ is the identity mapping on $E$. 
 Let $\{x_n\}$ be a sequence defined by $x_1 \in C$ and 
 \[
 x_{n+1}= \alpha_n u + (1-\alpha_n) J_{\lambda_n} x_n 
 \]
 for $n\in\N$, where $J_{\lambda_n} = (I+ \lambda_n A)^{-1}$. 
 Then $\{x_n\}$ converges strongly to $Qu$, 
 where $Q$ is the sunny nonexpansive retraction of $C$ onto
 $A^{-1}0$.
\end{theorem}

\begin{proof}
 Taking into account $\F(J_{\lambda_n}) = A^{-1}0$,
 we see that $\F(\{J_{\lambda_n}\}) = A^{-1}0 \ne \emptyset$. 
 We know that $\{ J_{\lambda_n}\}$ is a strongly nonexpansive sequence;
 see~\cite{MR2581778}*{Lemma~2.5}.
 We also know that $\{J_{\lambda_n}\}$ satisfies the NST condition (I)
 with $J_1$; see~\cite{MR2314663}*{Lemma 3.5}.
 Therefore Corollary~\ref{c:main} implies the conclusion.
\end{proof}

\subsection{Viscosity approximation method}

In the rest of this section, we deal with the viscosity approximation method
for strongly nonexpansive sequences. 
The viscosity approximation method for a single nonexpansive mapping was
originally proposed by Moudafi~\cite{MR1738332}; see
also~\cites{MR2086546, MR3185784, MR3213161, NMJ2016, MR2273529, MR2468414}
and references therein. 

Let $C$ be a subset of a Banach space $E$, $F$ a nonempty subset of $C$, 
$f\colon C\to C$ a mapping, and $\theta$ a real number in $[0,1)$.
Recall that $f$ is said to be a 
\emph{$\theta$-contraction} if
\begin{equation}\label{e:contraction}
\norm{f(x) - f(y)} \leq \theta \norm{x-y} 
\end{equation}
for all $x,y\in C$; $f$ is said to be a
\emph{$\theta$-contraction with respect to $F$}~\cites{MR3185784,
NMJ2016} if \eqref{e:contraction} holds for all $x\in C$ and $y\in F$. 

Applying Theorem~\ref{t:main}, we obtain the following theorem: 

\begin{theorem}\label{t:viscosity}
 Let $E$, $C$, $\{S_n\}$, $T$, $Q$, and $\{\alpha_n\}$ be the same as in
 Theorem~\ref{t:main},
 $\{f_n\}$ a sequence of self-mappings of $C$, 
 and $\{y_n\}$ a sequence defined by $y_1 \in C$ and 
 \begin{equation}\label{e:y_n}
  y_{n+1} = \alpha_n f_n (y_n) + (1-\alpha_n) S_n y_n    
 \end{equation}
 for $n\in\N$.
 Suppose that each $f_n$ is a $\theta$-contraction with respect to
 $\F(\{S_n\})$
 and there exists $u \in C$ such that $f_n (Qu)\to u$,
 where $\theta \in [0,1)$. 
 Then $\{y_n\}$ converges strongly to $Qu$. 
\end{theorem}

In order to prove Theorem~\ref{t:viscosity}, we need the following
lemma.
The proof is based on the technique developed in~\cite{MR2273529};
see also \cite{MR3185784}. 

\begin{lemma}\label{l:viscosity}
 Let $E$ be a Banach space, $C$ a nonempty closed convex subset of $E$,
 $F$ a nonempty sunny nonexpansive retract of $C$, 
 $Q$ a sunny nonexpansive retraction of $C$ onto $F$,
 $\{S_n\}$ a sequence of nonexpansive self-mappings of $C$,
 $\{f_n \}$ a sequence of self-mappings of $C$, 
 $u$ a point in $C$, 
 $\{\alpha_n\}$ a sequence in $[0,1]$, $\{x_n\}$ a sequence in $C$ defined
 by $x_1 \in C$ and \eqref{e:x_n} for $n \in \N$, and 
 $\{y_n\}$ a sequence in $C$ defined by $y_1 \in C$ and \eqref{e:y_n} for
 $n \in \N$. 
 Suppose that $\sum_{n=1}^\infty \alpha_n = \infty$,
 each $f_n$ is a $\theta$-contraction with respect to $F$, 
 $f_n (Qu) \to u$, and $x_n \to Qu$, where $\theta \in [0,1)$. 
 Then $\{y_n\}$ converges strongly to $Qu$. 
\end{lemma}

\begin{proof}
 Since $Qu \in F$ and $f_n$ is a $\theta$-contraction
 with respect to $F$, we have
 \begin{align*}
  \norm{u - f_n (y_n)}
  &\leq \norm{u - f_n (Qu)}
  + \norm{f_n (Qu) - f_n (y_n)}\\
  &\leq \norm{u - f_n (Qu)} + \theta \norm{Qu - y_n}\\
  &\leq \norm{u - f_n (Qu)}
  + \theta \norm{Qu - x_n} + \theta \norm{x_n - y_n}. 
 \end{align*}
 Thus we deduce from the nonexpansiveness of $S_n$ that
 \begin{align*}
  \norm{x_{n+1} - y_{n+1}}
  &\leq \alpha_n \norm{u - f_n (y_n)} + (1-\alpha_n) \norm{x_n - y_n}\\
  &\leq \bigl( 1 - (1-\theta)\alpha_n \bigr) \norm{x_n - y_n} \\
  &\qquad + (1-\theta)\alpha_n \left(
  \dfrac{\theta \norm{Qu - x_n} + \norm{u - f_n (Qu)}}
  {1-\theta} 
  \right)
 \end{align*}
 for all $n \in \N$. 
 By assumption, we know that
 $\sum_{n=1}^\infty (1-\theta)\alpha_n = \infty$, 
 $\norm{Qu - x_n}\to 0$, and $\norm{u - f_n (Qu)} \to0$.
 Thus Lemma~\ref{lm:seq} implies that $\norm{x_n - y_n} \to 0$,
 and hence $y_n \to Qu$.
\end{proof}

\begin{proof}[Proof of Theorem~\ref{t:viscosity}]
 Let $\{x_n\}$ be a sequence defined by $x_1 \in C$ and~\eqref{e:x_n}
 for $n \in \N$. Then it follows from Theorem~\ref{t:main} that $\{x_n\}$
 converges strongly to $Qu$.
 Thus Lemma~\ref{l:viscosity} implies the conclusion. 
\end{proof}

Using Theorem~\ref{t:viscosity}, we directly obtain the following two
corollaries:

\begin{corollary}
 Let $E$, $C$, $\{S_n\}$, $T$, $Q$, and $\{\alpha_n\}$ be the same as in
 Theorem~\ref{t:main}, $\{u_n \}$ a sequence in $C$, 
 and $\{y_n\}$ a sequence defined by $y_1 \in C$ and 
 \[
 y_{n+1} = \alpha_n u_n + (1-\alpha_n) S_n y_n    
 \]
 for $n\in\N$.
 Suppose that $u_n \to u$. Then $\{y_n\}$ converges strongly to $Qu$. 
\end{corollary}

\begin{proof}
 Let $f_n \colon C\to C$ be a mapping defined by $f_n (x) = u_n$ for all
 $x \in C$ and $n \in \N$.
 Then it is clear that each $f_n$ is a $0$-contraction with respect to
 $C$ and $f_n (Qu) = u_n \to u$.
 Therefore Theorem~\ref{t:viscosity} implies the conclusion. 
\end{proof}

\begin{corollary}
 Let $E$, $C$, $\{S_n\}$, $T$, $Q$, and $\{\alpha_n\}$ be the same as in
 Theorem~\ref{t:main},
 $\{f_n\}$ a sequence of self-mappings of $C$, 
 and $\{y_n\}$ a sequence defined by $y_1 \in C$ and 
 \eqref{e:y_n} for $n\in\N$.
 Suppose that each $f_n$ is a $\theta$-contraction with respect to
 $\F(\{S_n\})$ 
 and $\{f_n (z): n\in \N\}$ is a singleton for all $z\in \F(\{S_n\})$,
 where $\theta \in [0,1)$. 
 Then $\{y_n\}$ converges strongly to $w$, where $w$ is a unique fixed point
 of $Q\circ f_1$. 
\end{corollary}

\begin{proof}
 Since $Q \circ f_1$ is a $\theta$-contraction on $\F(\{S_n\})$, 
 $Q \circ f_1$ has a unique fixed point $w \in \F(\{S_n\})$. 
 Set $u = f_1(w)$.
 By assumption, $\{f_n (w): n \in \N\} = \{ u \}$.
 Thus we see that
 $f_n (Qu) = f_n \bigl( Q \bigl(f_1 (w) \bigr) \bigr) = f_n (w)  = u$
 for all $n \in \N$, and hence $f_n (Qu)  \to u$. 
 Therefore Theorem~\ref{t:viscosity} implies that
 $y_n \to Qu = Q \bigl( f_1 (w) \bigr)= w$. 
\end{proof}

\end{document}